\documentclass[reqno]{amsart}
\usepackage{amsmath}
\usepackage{arydshln}%
\allowdisplaybreaks[3]
\numberwithin{equation}{section}
\numberwithin{equation}{section}

\newtheorem{thm}{\indent Theorem}[section]
\newtheorem{cor}[thm]{\indent Corollary}
\newtheorem{lem}[thm]{\indent Lemma}

\newtheorem{dfn}{{\indent\bf Definition}}[section]
\newtheorem{rmk}{{\indent\bf Remark}}[section]
\newtheorem{expl}{{\indent\bf Example}}[section]

\newcommand{\mb}{\mbox}
\newcommand{\hs}{\hspace}
\newcommand{\dps}{\displaystyle}

\newcommand{\ttiny}{\fontsize{5pt}{\baselineskip}\selectfont}
\newcommand{\strl}[2]{\stackrel{\mbox{\ttiny $#1$}}{#2}}

\newcommand{\td}{\tilde}

\newcommand{\fr}{\frac}

\newcommand{\edd}{\end{document}}
\newcommand{\be}{\begin{equation}}
\newcommand{\ee}{\end{equation}}
\newcommand{\bsl}{\backslash}

\newcommand{\undbc}{\underbrace}

\newcommand{\lagl}{\langle}
\newcommand{\ragl}{\rangle}
\newcommand{\lmx}{\left(\begin{matrix}}
\newcommand{\rmx}{\end{matrix}\right)}
\newcommand{\ldt}{\left|\begin{matrix}}
\newcommand{\rdt}{\end{matrix}\right|}

\newcommand{\sgn}{{\rm Sgn\,}}

\newcommand{\tr}{{\rm tr\,}}

\newcommand{\veps}{\varepsilon}
\newcommand{\bbr}{{\mathbb R}}

\newcommand{\mo}{M\"obius }

\newcommand{\ba}{\begin{array}}
\newcommand{\ea}{\end{array}}
\newcommand{\nnm}{\nonumber}
\newcommand{\beal}{\begin{align}}
\newcommand{\eal}{\end{align}}
\newcommand{\bea}{\begin{eqnarray}}
\newcommand{\eea}{\end{eqnarray}}

\begin{document}

\title[Regular space-like hypersurfaces with parallel Blaschke tensorsin $\mathbb{S}^{m+1}_{1}$]{On the regular space-like hypersurfaces\\ in the de Sitter space $\mathbb{S}^{m+1}_{1}$ with parallel Blaschke tensors} 

\author[X. X. Li]{Xingxiao Li$^*$} 

\author[H. R. Song]{Hongru Song} 

\dedicatory{}

%
\thanks{Research supported by
Foundation of Natural Sciences of China (No. 11171091, 11371018).}
\address{
School of Mathematics and Information Sciences
\endgraf Henan Normal University \endgraf Xinxiang 453007, Henan
\endgraf P.R. China}
\email{xxl$@$henannu.edu.cn}

\address{
School of Mathematics and Information Sciences
\endgraf Henan Normal University \endgraf Xinxiang 453007, Henan
\endgraf P.R. China} %
\email{yaozheng-shr$@$163.com}



\begin{abstract}
In this paper, we use two conformal non-homogeneous coordinate systems, modeled on the de Sitter space $\mathbb{S}^{m+1}_1$, to cover the conformal space $\mathbb{Q}^{m+1}_1$, so that the conformal geometry of regular space-like hypersurfaces in $\mathbb{Q}^{m+1}_1$ is treated as that of hypersurfaces in  $\mathbb{S}^{m+1}_1$. As a result, we give a complete classification of the regular space-like hypersurfaces (represented in the de Sitter space $\mathbb{S}^{m+1}_1$) with parallel Blaschke tensors.
\end{abstract}

\maketitle

\section{Introduction} 

Let $\bbr^{s+m}_s$ be the $(s+m)$-dimensional pseudo-Euclidean space which is the real vector space $\bbr^{s+m}$ equipped with the non-degenerate inner product $\lagl\cdot,\cdot\ragl_s$ given by
$$
\lagl X,Y\ragl_s=-X_1\cdot Y_1+X_2\cdot Y_2,\quad
X=(X_1,X_2),\,Y=(Y_1,Y_2)\in\bbr^s\times\bbr^m\equiv\bbr^{s+m}.
$$
where the dot ``$\cdot$'' is the standard Euclidean inner product either on $\bbr^{s}$ or on $\bbr^{m}$.

Denote by $\mathbb{RP}^{m+2}$ the real projection space of dimension $m+2$. Then the so called conformal space $\mathbb{Q}^{m+1}_1$ is defined as (\cite{ag1})
$$
\mathbb{Q}^{m+1}_1={\{[\xi]\in \mathbb{RP}^{m+2};\lagl \xi,\xi\ragl_2=0\}},
$$
while, for any $a>0$, the de Sitter space $\mathbb{S}^{m+1}_1(a)$ and the anti-de Sitter space $\mathbb{H}^{m+1}_1\left(-\fr1{a^2}\right)$ are defined respectively by
$$
\mathbb{S}^{m+1}_1(a)=\{\xi\in\bbr^{m+2}_1;\lagl \xi,\xi \ragl_1=a^2\},\quad \mathbb{H}^{m+1}_1\left(-\fr1{a^2}\right)=\{\xi\in\bbr^{m+2}_2;\lagl \xi,\xi \ragl_2=-a^2\}.
$$
Then $\mathbb{S}^{m+1}_1(a)$, $\mathbb{H}^{m+1}_1\left(-\fr1{a^2}\right)$ and the Lorentzian space $\bbr^{m+1}_1$ are called Lorentzian space forms. Denote $\mathbb{S}^{m+1}_1=\mathbb{S}^{m+1}_1(1)$ and $\mathbb{H}^{m+1}_1=\mathbb{H}^{m+1}_1\left(-1\right)$. Define three hyperplanes as follows:
\begin{align*}
\pi=&\{[x]\in \mathbb{Q}^{m+1}_1;x_1=x_{m+2}\},\\\pi_+=&\{[x]\in \mathbb{Q}^{m+1}_1;x_{m+2}=0\},\\ \pi_-=&\{[x]\in \mathbb{Q}^{m+1}_1;x_1=0\}.
\end{align*}
Then there are three conformal diffeomorphisms from the Lorentzian space forms into the conformal space:
\begin{equation}\label{eq1.1}
\begin{aligned}
&\sigma_0:\bbr^m_1\to \mathbb{Q}^{m+1}_1\bsl\pi,\quad u\longmapsto\left[\left(\lagl u,u\ragl_1-1,2u,\lagl u,u\ragl_1+1\right)\right],\\
&\sigma_1:\mathbb{S}^{m+1}_1\to \mathbb{Q}^{m+1}_1\bsl\pi_+,\quad u\longmapsto\left[\left(u,1\right)\right],\\
&\sigma_{-1}:\mathbb{H}^{m+1}_1\to \mathbb{Q}^{m+1}_1\bsl\pi_-,\quad u\longmapsto\left[\left(1,u\right)\right].
\end{aligned}
\end{equation}
Therefore $\mathbb{Q}^{m+1}_1$ is the common conformal compactification of $\bbr^{m+1}_1$, $\mathbb{S}^{m+1}_1$ and $\mathbb{H}^{m+1}_1$.

In the reference \cite{ag1}, Nie at al successfully set up a unified framework of conformal geometry for both regular surfaces and hypersurfaces in Lorentzian space forms by introducing the conformal space $\mathbb{Q}^{m+1}_1$ and some basic conformal invariants, including the conformal metric $g$, the conformal form $\Phi$, the Blaschke tensor $A$ and the conformal second fundamental form $B$. Later, all of these have been generalized to regular submanifolds of higher codimensions (\cite{ag3}). Under this framework, several characterization or classification theorems are obtained for hypersurfaces with some special conformal invariants, see for example (\cite{ag1},\cite{ag2}). The achievement of these certainly proves the efficiency of the above framework. In particular, as the main theorems, regular hypersurfaces with parallel conformal second fundamental forms, and conformal isotropic submanifolds are classified in \cite{ag1} and \cite{ag3}, respectively. Note that, a regular submanifold in the conformal space $\mathbb{Q}^{m+1}_1$ with vanishing conformal form is called conformal isotropic if its Blaschke tensor $A$ is parallel to the conformal metric. For the later use, we rewrite these two theorems applied in the special case of space-like hypersurfaces as follows:

\begin{thm}[\cite{ag1}]\label{nie1} Let $x:M^m\to \mathbb{Q}^{m+1}_1$
be a regular space-like hypersurface with parallel conformal second
fundamental form. Then $x$ is locally conformal equivalent to one of the
following hypersurfaces:
\begin{enumerate}
\item $\mathbb{H}^k\times \bbr^{m-k} \subset \bbr^{m+1}_1$, \quad $k=1,\cdots m-1$; or
\item $\mathbb{S}^{m-k}(a)\times\mathbb{H}^k\left(-\fr1{a^2-1}\right) \subset \mathbb{S}^{m+1}_1$ , $a>1, k=1,\cdots m-1$; or
\item $\mathbb{H}^k\left(-\fr1{a^2}\right)\times \mathbb{H}^{m-k}(-\fr1{1-a^2}) \subset \mathbb{H}^{m+1}_1$ , $0<a<1, k=1,\cdots m-1$; or
\item $WP(p,q,a)\subset \bbr^{m+1}_1$ for some constants $p,q,a$, as indicated in
Example $3.1$.
\end{enumerate}
\end{thm}

\begin{thm}[\cite{ag3}]\label{nie2}
Any regular, space-like and conformal isotropic hypersurface in $\mathbb{Q}^{m+1}_1$ is conformal equivalent to a maximal, space-like and regular hypersurface in $\bbr^{m+1}_1$, $\mathbb{S}^{m+1}_1$ or $\mathbb{H}^{m+1}_1$ with constant scalar curvature.
\end{thm}

We remark that a \mo classification of umbilic-free hypersurfaces in the unit sphere with parallel \mo second fundamental forms has been established in \cite{hl04}.

Motivated by the above theorems, we aims in the present paper at a complete classification of regular space-like hypersurfaces in $\mathbb{Q}^{m+1}_1$ with parallel Blaschke tensors. To this end, we would like to make a direct use of the ideas and technics with which we previously studied the M\"obius geometry of umbilc-free hypersurfaces in the unit sphere(\cite{hlz}, \cite{lz1}, \cite{lz2} and \cite{lz3}). So we firstly define two conformal non-homogeneous coordinate systems (with the coordinate maps $\strl{(1)}{\Psi}$, $\strl{(2)}{\Psi}$ respectively) covering the conformal space $\mathbb{Q}^{m+1}_1$, which are modeled on the de Sitter space $\mathbb{S}^{m+1}_1$, so that the conformal geometry of the hypersurfaces in  $\mathbb{Q}^{m+1}_1$ corresponds right to that of the hypersurfaces in the de Sitter space. It follows that the conformal geometry of regular hypersurfaces in each of $\mathbb{H}^{m+1}_1$ and $\mathbb{R}^{m+1}_1$ is made unified with that in $\mathbb{S}^{m+1}_1$. This shows that we only need to consider and study the conformal invariants of the hypersurfaces in $\mathbb{S}^{m+1}_1$ which plays the same role as the unit sphere does in the M\"obius geometry of umbilic-free submanifolds. With this consideration, we only focus here on the study of the conformal invariants of regular space-like hypersurfaces in the de Sitter space $\mathbb{S}^{m+1}_1$. As a result, we are able to establish a complete classification for all the regular space-like hypersurfaces with parallel Blaschke tensors.

Note that the above two conformal non-homogeneous coordinate maps $\strl{(1)}{\Psi}$ and $\strl{(2)}{\Psi}$ are conformal equivalent whereon both of them are defined. Therefore we can use $\Psi$ to denote either one of $\strl{(1)}{\Psi}$ and $\strl{(2)}{\Psi}$. By this, the main theorem of the present paper is stated as follows:

\begin{thm}\label{main} Let $x:M^m\to \mathbb{S}^{m+1}_1$, $m\geq 2$, be a regular space-like hypersurface. If the Blaschke tensor $A$ of $x$ is
parallel, then one of the following holds:
\begin{enumerate}
\item $x$ is conformal isotropic and thus is locally conformal equivalent
to a maximal space-like regular hypersurface in $\mathbb{S}^{m+1}_1$ with
constant scalar curvature, or the conformal image under $\Psi\circ\sigma_{-1}$ of a maximal regular hypersurface in $\mathbb{H}^{m+1}_1$ with
constant scalar curvature, or the conformal image under $\Psi\circ\sigma_0$ of a maximal regular hypersurface in $\mathbb{R}^{m+1}_1$ with
constant scalar curvature;
\item $x$ is of parallel conformal second fundamental form $B$ and thus is
locally conformal equivalent to
\begin{enumerate}
\item the image under $\Psi\circ\sigma_0$ of $\mathbb{H}^k\times \bbr^{m-k} \subset \bbr^{m+1}_1$, \quad $k=1,\cdots m-1$; or
\item $\mathbb{S}^{m-k}(a)\times\mathbb{H}^k\left(-\fr1{a^2-1}\right) \subset \mathbb{S}^{m+1}_1$ , $a>1, k=1,\cdots m-1$; or
\item the image under $\Psi\circ\sigma_{-1}$ of $\mathbb{H}^k\left(-\fr1{a^2}\right)\times \mathbb{H}^{m-k}(-\fr1{1-a^2}) \subset \mathbb{H}^{m+1}_1$ , $0<a<1, k=1,\cdots m-1$; or
\item $WP(p,q,a)\subset \bbr^{m+1}_1$ for some constants $p,q,a$.
\end{enumerate}
\item $x$ is non-isotropic with a non-parallel conformal second
fundamental form $B$ and is locally conformal equivalent to
\begin{enumerate}
\item one of the maximal hypersurfaces as indicated in
Example \ref{expl3.2}, or
\item one of the non-maximal hypersurfaces as indicated in
Example \ref{expl3.3}.
\end{enumerate}
\end{enumerate}
\end{thm}

\begin{rmk}\label{rem1}\rm It is directly verified in Section \ref{sec3} that each of the regular space-like
hypersurfaces stated in the above theorem has a
parallel Blaschke tensor.
\end{rmk}

\section{Necessary basics on regular space-like hypersurfaces}\label{sec2}

This section provides some basics of the conformal geometry of regular space-like hypersurfaces in the Lorentzian space forms. The main idea comes originally from the work of C.P. Wang on the M\"obius geometry of umbilic-free submanifolds in the unit sphere (\cite{wcp}), and much of the detail can be found in a series of papers by Nie at al (see for example \cite{ag1}, \cite{ag2}, \cite{ag3}).

Let ${x}:M^m\to \mathbb{S}^{m+1}_{1}\subset\bbr^{m+2}_1$ be a regular space-like hypersurface in $\mathbb{S}^{m+1}_{1}$.
Denote by  ${h}$ the (scalar-valued) second fundamental form of ${x}$
with components ${h}_{ij}$ and ${H}=\fr1m\tr {h}$
the mean curvature. Define the conformal factor $\rho>0$ and the conformal position $ Y$ of ${x}$, respectively, as follows:
\begin{equation}\label{eq2.1}
{\rho}^2=\fr m{m-1}\left(|{h}|^2-m|{H}|^2\right),\quad {Y}={\rho}(1, x)\in \bbr^1_1\times\bbr^{m+2}_1\equiv\bbr^{m+3}_2.
\end{equation}
Then ${Y}(M^m)$ is clearly included in the light cone $\mathbb{C}^{m+2}\subset\bbr^{m+3}_2$ where
$$
\mathbb{C}^{m+2}=\{\xi\in\bbr^{m+3}_2;\lagl \xi,\xi\ragl_2=0,\xi\neq0\}.
$$
The positivity of $\rho$ implies that ${Y}:M^m\to \bbr^{m+3}_2$ is an immersion of $M^m$ into the $\bbr^{m+3}_2$.
Clearly, the metric ${g}:=\lagl d{Y},d{Y}\ragl_2\equiv{\rho}^2\lagl dx,dx\ragl_1$ on $M^m$, induced by $Y$ and called the conformal metric, is invariant under the pseudo-orthogonal group $O(m+3,2)$ of linear transformations on $\bbr^{m+3}_2$ reserving the Lorentzian product $\lagl\cdot,\cdot\ragl_2$. Such kind of things are called the conformal invariants of $ x$.

{\dfn[cf. \cite{ag1},\cite{ag2},\cite{ag3}] Let $x,\td x:M^m\to \mathbb{S}^{m+1}_{1}$ be two regular space-like hypersurfaces with $Y,\td Y$ their conformal positions, respectively. If there exists some $\mathbb{T}\in O(m+3,2)$ such that $\td Y=\mathbb{T}(Y)$, then $x,\td x$ are called conformal equivalent to each other.}

For any local orthonormal frame field $\{e_i\}$ and the dual $\{\theta^i\}$ on $M^m$ with respect to the standard metric $\lagl d x,d x\ragl_1$, define
\begin{equation}
E_i={\rho}^{-1}e_i,\quad \omega^i={\rho}\theta^i.
\end{equation}
Then $\{E_i\}$ is a local orthonormal frame field with respect to the conformal metric ${g}$ with $\{\omega^i\}$ its dual coframe.
Let ${n}$ be the time-like unit normal of ${x}$. Define
$$ {\xi}=\left(-{H},-{H} x+{n}\right),$$
then $\lagl {\xi},{\xi} \ragl_2=-1$.
Let $\Delta$ denote the Laplacian with respect to  the conformal metric ${g}$. Define
${N}:M^{m}\to\bbr^{m+3}_2$ by
\begin{equation}
{N}=-\fr1m\Delta {Y}-\fr1{2m^2}\lagl\Delta {Y},\Delta {Y}\ragl_2{Y}.
\end{equation}
Then it holds that
\begin{equation}
\lagl \Delta {Y},{Y}\ragl_2=-m,\quad
\lagl {Y},{Y}\ragl_2=\lagl {N},{N}\ragl_2=0,\quad \lagl {Y},{N}\ragl_2=1.
\end{equation}
Furthermore, $\{{Y},{N},{Y}_{i},{\xi},\  1\leq i\leq m\}$ forms a moving frame in $\bbr^{m+3}_2$ along ${Y}$, with respect to which the equations of motion is as follows:
\begin{equation}\label{eq2.5}
\left\{\begin{aligned}
d{Y}=&\,\sum{Y}_i\omega^i,\\
d{N}=&\,\sum\psi_i{Y}_i+\phi  \xi,\\
d{Y}_i=&-\psi_i{Y}-\omega_i {N}+\sum\omega_{ij}{Y}_j+\tau_i  \xi,\\
d \xi=&\,\phi {Y}+\sum\tau_i{Y}_i.
\end{aligned}\right.
\end{equation}
By the exterior differentiation of \eqref{eq2.5} and using Cartan's lemma, we can write
\begin{equation}
\begin{aligned}
&\phi=\sum_i\Phi_i\omega^i,\quad \psi_i=\sum_jA_{ij}\omega^j,\quad A_{ij}=A_{ji};\\
&\tau_i=\sum_jB_{ij}\omega^j,\quad B_{ij}=B_{ji}.
\end{aligned}
\end{equation}
Then the conformal form $\Phi$, the Blaschke tensor $A$ and the conformal second fundamental form $B$ defined by
$$\Phi=\sum_i\Phi_i\omega^i,\quad A=\sum_{i,j}A_{ij}\omega^i\omega^j,\quad B=\sum_{i,j}B_{ij}\omega^i\omega^j$$
are all conformal invariants. By a long but direct computation, we find that
\begin{align}
A_{ij}=&-\lagl{Y}_{ij},{N}\ragl_2=-{\rho}^{-2}((\log {\rho})_{,ij}-e_i(\log {\rho})e_j(\log {\rho})+{h}_{ij}{H})\nnm\\
&-\fr12{\rho}^{-2}\left(|\bar{\nabla}\log {\rho}|^2 -|H|^2-1\right)\delta_{ij},\\
B_{ij}=&-\lagl{Y}_{ij},{\xi}\ragl_2={\rho}^{-1}({h}_{ij}-{H}\delta_{ij}),\label{eq2.8}\\
\Phi_i=&-\lagl {\xi},d{N}\ragl_2=-{\rho}^{-2}[({h}_{ij}-{H}\delta_{ij})e_j(\log {\rho})+e_i({H})],
\end{align}
where $Y_{ij}=E_j(Y_i)$, $\bar\nabla$ is the Levi-Civita connection of the standard metric $\lagl\cdot,\cdot\ragl_1$, and the subscript ${}_{,ij}$ denotes the covariant derivatives with respect to $\bar\nabla$. The differentiation of \eqref{eq2.5} also gives the following integrability conditions:
\begin{align}
&\Phi_{ij}-\Phi_{ji}
=\sum(B_{ik}A_{kj}-B_{kj}A_{ki}),\label{Phiij}\\
&A_{ijk}-A_{ikj}=B_{ij}\Phi_k-B_{ik}\Phi_j,\label{Aijk}\\
&B_{ijk}-B_{ikj}=\delta_{ij}\Phi_k-\delta_{ik}\Phi_j,\label{Bijk}\\
&R_{ijkl}=\sum(B_{ik}B_{jl}
-B_{il}B_{jk})+A_{il}\delta_{jk}-A_{ik}\delta_{jl}
+A_{jk}\delta_{il}-A_{jl}\delta_{ik},\label{Rijkl}
\end{align}
where $A_{ijk}$, $B_{ijk}$, $\Phi_{ij}$ are respectively the components of the covariant derivatives of $A$, $B$, $\Phi$, and $R_{ijkl}$ is the components of the Riemannian curvature tensor of the conformal metric $g$. Furthermore, by \eqref{eq2.1} and \eqref{eq2.8} we have
\be\label{trB}
\tr B=\sum B_{ii}=0,\quad |B|^2=\sum(B_{ij})^2=\fr{m-1}{m},
\ee
and by \eqref{Rijkl} we find the Ricci curvature tensor
\be\label{Rij}
R_{ij}=\sum B_{ik}B_{kj}+\delta_{ij}\tr
A+(m-2)A_{ij},
\ee
which implies that
\be\label{trA}
\tr A=\fr{1}{2m}(m^2\kappa-1)
\ee
with $\kappa$ being the normalized scalar curvature of $g$.

It is easily seen (\cite{ag1}) that the conformal position vector ${Y}$ defined above is exactly the canonical lift of the composition map
$\bar x=\sigma_1\circ{x}:M^m\to\mathbb{Q}^{m+1}_1$, implying that the conformal invariants ${g},\Phi,A,B$ defined above are the same as those of $\bar x$ introduced by Nie at al in \cite{ag1}.

On the other hand, the conformal space $\mathbb{Q}^{m+1}_1$ is clearly covered by the following two open sets:
\begin{equation}
\begin{aligned}
U_1=\left\{[y]\in\mathbb{Q}^{m+1}_1;y=(y_1,y_2,y_3)\in \bbr^1_1\times\bbr^1_1\times\bbr^{m+1}\equiv\bbr^{m+3}_2,y_1\neq0\right\},\\
U_2=\left\{[y]\in\mathbb{Q}^{m+1}_1;y=(y_1,y_2,y_3)\in \bbr^1_1\times\bbr^1_1\times\bbr^{m+1}\equiv\bbr^{m+3}_2,y_2\neq0\right\}.
\end{aligned}
\end{equation}
Define the following two diffeomorphisms
\begin{equation}
\strl{(\alpha)}{\Psi}:U_{\alpha}\to \mathbb{S}^{m+1}_1,\quad \alpha=1,2
\end{equation}
by
\begin{align}
\strl{{ (1)}}{\Psi}([y])=y^{-1}_1(y_2,y_3), \text{ for } [y]\in U_1,\ y=(y_1,y_2,y_3);\\
\strl{{ (2)}}{\Psi}([y])=y^{-1}_2(y_1,y_3), \text{ for } [y]\in U_2,\ y=(y_1,y_2,y_3). \end{align}
Then, with respect to the conformal structure on $\mathbb{Q}^{m+1}_1$ introduced in \cite{ag1} and the standard metric on $\mathbb{S}^{m+1}_1$, both $\strl{(1)}{\Psi}$ and $\strl{(2)}{\Psi}$ are conformal.

Now for a regular space-like hypersurface $\bar x:M^m \to\mathbb{Q}^{m+1}_1$ with the canonical lift $$Y:M^m\to\mathbb{C}^{m+2}\subset\bbr^{m+3}_2,$$
write $Y=(Y_1,Y_2,Y_3)\in \bbr^1_1\times\bbr^1_1\times\bbr^{m+1}$. Then we have the following two composed hypersurfaces:
\begin{equation}
\strl{{ (\alpha)}}{x}:=\strl{{ (\alpha)}}{\Psi}\! \circ \,\,\bar x|_{\strl{{ (\alpha)}}{M}}:\strl{{ (\alpha)}}{M}\to \mathbb{S}^{m+1}_1, \quad
\strl{{ (\alpha)}}{M}=\{p\in M; x(p)\in U_{\alpha}\},\quad\alpha=1,2.
\end{equation}
Then $M^m=\strl{(1)}{M}\bigcup \strl{(2)}{M}$, and the following lemma is clearly true by a direct computation:

\begin{lem} The conformal position vector $\strl{{ (1)}}{Y}$ of $\strl{{ (1)}}{x}$ is nothing but $Y|_{\strl{{ (1)}}{M}}$, while the conformal position vector $\strl{{ (2)}}{Y}$ of $\strl{{ (2)}}{x}$ is given by
\begin{equation}
\strl{{ (2)}}{Y}=\mathbb{T}(Y|_{\strl{{(2)}}{M}}),\mbox{ where } \mathbb{T}=\left( \begin{tabular}{c:c}$\begin{matrix} 0&1\\1&0\end{matrix}$&$0$\\
\hdashline $0$&$I_{m+1}$\end{tabular}\right).
\end{equation}
\end{lem}
\begin{cor}\label{cor2.2}
The basic conformal invariants $g,\Phi,A,B$ of $\bar x$ coincide accordingly with those of each of $\strl{{ (1)}}{x}$ and $\strl{{ (2)}}{x}$ on where $\strl{{ (1)}}{x}$ or $\strl{{ (2)}}{x}$ is defined, respectively.
\end{cor}

Therefore, $\strl{{ (1)}}{\Psi}$ and $\strl{{ (2)}}{\Psi}$ can be viewed as two non-homogenous coordinate maps preserving the conformal invariants of the regular space-like hypersurfaces.

\begin{cor}
$\strl{{ (1)}}{x}$ and $\strl{{ (2)}}{x}$ are conformal equivalent to each other on $\strl{{ (1)}}{M}\cap \strl{{ (2)}}{M}$.
\end{cor}

On the other hand, all the regular space-like hypersurfaces in the three Lorentzian space forms can be viewed as ones in $\mathbb{Q}^{m+1}_1$ via the conformal embeddings $\sigma_1$, $\sigma_0$ and $\sigma_{-1}$ defined in \eqref{eq1.1}. Now, using $\strl{{ (1)}}{\Psi}$ and $\strl{{ (2)}}{\Psi}$, one can shift the conformal geometry of regular space-like hypersurfaces in $\mathbb{Q}^{m+1}_1$ to that of regular space-like hypersurfaces in the de Sitter space $\mathbb{S}^{m+1}_1$. It follows that, in a sense, the conformal geometry of regular space-like hypersurfaces can also be unified as that of the corresponding hypersurfaces in the de Sitter space.
Concisely, we can achieve this simply by introducing the following four conformal maps:
\begin{align}
&\strl{{ (1)}}{\sigma}=\strl{{ (1)}}{\Psi}\!\circ \,\,\sigma_0:\ \strl{{ (1)}}{\bbr}\!\!{}^{m+1}_1\rightarrow \mathbb{S}^{m+1}_1 ,\quad
u\mapsto\left(\fr{2u}{1+\lagl u,u\ragl},\fr{1-\lagl u,u\ragl}{1+\lagl u,u\ragl}\right),
\\
&\strl{{ (2)}}{\sigma}=\strl{{ (2)}}{\Psi}\!\circ \,\,\sigma_0:\ \strl{{ (2)}}{\bbr}\!\!{}^{m+1}_1\rightarrow \mathbb{S}^{m+1}_1 ,\quad
u\mapsto\left(\fr{1+\lagl u,u\ragl}{2u_1},\fr{u_2}{u_1},\fr{1-\lagl u,u\ragl}{2u_1}\right),
\\
&\strl{{ (1)}}{\tau}=\strl{{ (1)}}{\Psi}\!\circ \,\,\sigma_{-1}:\ \strl{{ (1)}}{\mathbb{H}}\!\!{}^{m+1}_1\rightarrow \mathbb{S}^{m+1}_1 ,\quad
y\mapsto\left(\fr{y_2}{y_1},\fr{y_3}{y_1},\fr{1}{y_1}\right),
\\
&\strl{{ (2)}}{\tau}=\strl{{ (2)}}{\Psi}\!\circ \,\,\sigma_{-1}:\ \strl{{ (2)}}{\mathbb{H}}\!\!{}^{m+1}_1\rightarrow \mathbb{S}^{m+1}_1 ,\quad
y\mapsto\left(\fr{y_1}{y_2},\fr{y_3}{y_2},\fr{1}{y_2}\right),
\end{align}
where
\begin{align}
&\strl{{ (1)}}{\bbr}\!\!{}^{m+1}_1=\{u\in \bbr^{m+1}_1;\quad 1+\lagl u,u \ragl\neq0\},\\
&\strl{{ (2)}}{\bbr}\!\!{}^{m+1}_1=\{u=(u_1,u_2)\in \bbr^{m+1}_1;u_1\neq0\},\\
&\strl{{ (1)}}{\mathbb{H}}\!\!{}^{m+1}_1=\{y=(y_1,y_2,y_3)\in \mathbb{H}^{m+1}_1;y_1\neq0\},\\
&\strl{{ (2)}}{\mathbb{H}}\!\!{}^{m+1}_1=\{y=(y_1,y_2,y_3)\in \mathbb{H}^{m+1}_1;y_2\neq0\}.
\end{align}

The following theorem will be used later in this paper:

\begin{thm}[\cite{ag3}]\label{thm2.4} Two hypersurfaces $x:M^m \rightarrow \mathbb{S}^{m+1}_p$ and $\td x:\td{M}^m \rightarrow \mathbb{S}^{m+1}_p$ $(m\geq3,p\geq0)$ are conformal equivalent if and only if there exists a diffeomorphism $f:M\rightarrow \td{M}$ which preserves the conformal metric and the conformal second fundamental form.
\end{thm}

\section{Examples}\label{sec3}

Before proving the main theorem, we first present some regular space-like hypersurfaces in $\mathbb{S}^{m+1}_1$ with parallel Blaschke tensors.

{\expl[\cite{ag1}, cf.\cite{hl04}]\label{expl3.1}\rm Let $\bbr^+$ be the half line of
positive real numbers. For any two given natural numbers $p,q$
with $p+q<m$ and a real number $a>1$, consider the hypersurface of warped product embedding
$$\dps u:\mathbb{H}^q\left(-\fr1{a^2-1}\right)\times\mathbb{S}^p(a)\times \bbr^+\times\bbr^{m-p-q-1}\to
\bbr^{m+1}_1$$
defined by
$$u=(tu',tu'',u'''),\mb{ where }
u'\in \mathbb{H}^q\left(-\fr1{a^2-1}\right),\ u''\in \mathbb{S}^p(a),\ t\in\bbr^+,\ u'''\in\bbr^{m-p-q-1}.$$
Then $\bar x:=\sigma_0\circ u$ is a regular space-like hypersurface in the conformal space $\mathbb{Q}^{m+1}_1$ with parallel conformal second fundamental form. This hypersurface is denoted as $WP(p,q,a)$ in \cite{ag1}. By Proposition 3.1 together with its proof in \cite{ag1}, $\bar x$ is also of parallel Blaschke tensor. It follows from Corollary \ref{cor2.2} that the composition map $$x=\Psi\circ\bar x:\mathbb{H}^q\left(-\fr1{a^2-1}\right)\times\mathbb{S}^p(a)\times \bbr^+\times\bbr^{m-p-q-1}\to \mathbb{S}^{m+1}_1,$$
where $\Psi$ denotes $\strl{(1)}{\Psi}$ or $\strl{(2)}{\Psi}$, defines a regular space-like hypersurface in $\mathbb{S}^{m+1}_1$ with both parallel conformal second fundamental form and parallel Blaschke tensor. For convenience, we also denote $x$ by the same symbol $WP(p,q,a)$. Note that, by a direct calculation, one easily finds that $WP(p,q,a)$ has exactly
three distinct conformal principal curvatures.

The similar example of $WP(p,q,a)$ in M\"obius geometry was originally found by \cite{hl04} and denoted by $CSS(p,q,a)$.}

{\expl\label{expl3.2}\rm Given $r>0$. For any integers $m$ and $K$ satisfying $m\geq
3$ and $2\leq K\leq m-1$, let $\td y_1:M^{K}_1\to
\mathbb{S}^{K+1}_1(r)\subset\bbr^{K+2}_1$ be a regular and maximal space-like hypersurface with constant scalar curvature}
\begin{equation}\label{ex1S}
\td S_1=\fr{mK(K-1)+(m-1)r^2}{mr^2}
\end{equation}
and
$$
\td y=(\td y_0,\td y_2):
\mathbb{H}^{m-K}\left(-\fr1{r^2}\right)\to\bbr^1_1\times\bbr^{m-K}\equiv\bbr^{m-K+1}_1
$$
be the canonical embedding, where $\td y_0>0$. Set
\begin{equation}\label{eq3.2}
\td M^m=M^{K}_1\times \mathbb{H}^{m-K}\left(-\fr1{r^2}\right),\quad\td Y=(\td
y_0,\td y_1,\td y_2).
\end{equation}
Then $\td Y:\td M^m\to\bbr^{m+3}_2$ is an immersion satisfying
$\lagl \td Y,\td Y\ragl_2=0$. The induced metric
$$g=\lagl d\td Y,d\td Y\ragl_2=-d\td y_0^2+\lagl d\td y_1,d\td y_1\ragl_1+d\td y_2\cdot d\td y_2$$
by $\td Y$ is clearly a Riemannian one, and thus as Riemannian manifolds we have
\begin{equation}
\label{exm1}
(\td M^m,g)=(M_1,\lagl d\td y_1,d\td y_1\ragl_1)\times \left(\mathbb{H}^{m-K}\left(-\fr1{r^2}\right),
\lagl d\td y,d\td y\ragl_1\right).
\end{equation}
Define
\begin{equation}
\td x_1=\fr{\td y_1}{\td y_0},\quad \td
x_2=\fr{\td y_2}{\td y_0}, \quad \td x=(\td x_1,\td x_2).
\end{equation}
Then $\td x^2=1$ and thus we have a smooth map $\td x:M^m\to \mathbb{S}^{m+1}_1$. Clearly,
\begin{equation}
\label{dx}
d\td x=-\fr{d\td y_0}{\td y^2_0}(\td y_1,\td y_2) +\fr1{\td
y_0}(d\td y_1,d\td y_2).
\end{equation}
Therefore the induced ``metric'' $\td g=d\td x\cdot d\td x$ is derived as
\begin{align}
\td g=\ &\td y^{-4}_0d\td y^2_0(\lagl \td y_1,\td y_1\ragl_1+\td y_2\cdot \td y_2)+\td y^{-2}_0(\lagl d\td y_1,d\td y_1\ragl_1+d\td y_2\cdot d\td y_2)\\
&-2\td y^{-3}_0d\td y_0(\lagl \td y_1,d\td y_1\ragl_1+\td y_2\cdot d\td y_2)\\
=\ &\td y^{-2}_0(d\td y^2_0+g+d\td y^2_0-2d\td y^2_0)\\
=\ &\td y^{-2}_0g,
\end{align}
implying that $\td x$ is a regular space-like hypersurface.

If $\td n_1$ is the time-like unit normal vector field of $\td y_1$ in
$\mathbb{S}^{K+1}_1(r)\subset\bbr^{K+2}_1$, then $\td n=(\td
n_1,0)\in\bbr^{m+2}_1$ is a time-like unit normal vector field of $\td x$.
Consequently, by \eqref{dx}, the second fundamental form $\td h$ of $\td
x$ is given by
\begin{equation}
\begin{aligned}
\td h=\lagl d\td n, d\td x\ragl_1&=\lagl(d\td n_1,0),-\td y^{-2}_0d\td y_0(\td y_1,\td y_2)+\td y^{-1}_0(d\td y_1,d\td y_2)\ragl_1\\
&=-\td y^{-2}_0d\td y_0\lagl d\td n_1,\td y_1\ragl_1+\td y^{-1}_0\lagl d\td n_1,d\td y_1\ragl_1\\
& =\td y^{-1}_0h,
\end{aligned}
\end{equation}
where $h$ is the second fundamental form of $\td y_1:M^{K}_1\rightarrow\mathbb{S}^{K+1}_1$.

Let $\{E_i\,;1\leq i\leq K\}$ (resp. $\{E_i\,;K+1\leq i\leq
m\}$) be a local orthonormal frame field on $(M_1,d\td y^2_1)$
(resp. on $\mathbb{H}^{m-K}(-\fr1{r^2})$). Then $\{E_i\,;1\leq i\leq m\}$
gives a local orthonormal frame field on $(\td M^m,g)$. Put $e_i=\td
y_0E_i$, $i=1,\cdots,m$. Then $\{e_i\,;1\leq i\leq m\}$ is a orthonormal frame field along $\td x$.
 Thus we obtain
\begin{equation}
\td h_{ij}=\td h(e_i,e_j)=\td y^2_0\td h(E_i,E_j)=\begin{cases}\td y_0
h(E_i,E_j)=\td y_0h_{ij},\quad &\mb{when\ } 1\leq i,j\leq K,\\0,\quad &\mb{otherwise\ }.
\end{cases}
\end{equation}
Since the mean curvature of $\td y_1\equiv 0$ by the maximality of $\td y_1$, the mean
curvature $\td H$ of $\td x$ vanishes. Therefore
$$\td \rho^2=\fr{m}{m-1}\left(\sum_{i,j}\td h^2_{ij}-m|\td
H|^2\right)=\fr{m}{m-1}\td y^2_0\sum_{i,j=1}^{K}h_{ij}^2 =\td
y^2_0,$$
where we have used the Gauss equation and \eqref{ex1S}. It follows that $\td x$ is regular and its conformal factor $\td
\rho=\td y_0$. Thus $\td Y$, given in \eqref{eq3.2}, is exactly the conformal position vector of $\td x$, implying the induced metric $g$ by $\td Y$ is nothing but
the conformal metric of $\td x$. Furthermore, the conformal second fundamental form of
$\td x$ is given by
\begin{equation}
\label{ex1B}
\td B=\td \rho^{-1}\sum(\td h_{ij}-\td H\delta_{ij})\omega^i\omega^j
=\sum_{i,j=1}^{K}h_{ij}\omega^i\omega^j,
\end{equation}
where $\{\omega^i\}$ is the local coframe field on $M^m$ dual to
$\{E_i\}$.

On the other hand, by \eqref{exm1} and the Gauss equations of $\td y_1$ and
$\td y$, one finds that the Ricci tensor of $g$ is given as follows:
\begin{align}
R_{ij}=&\fr{K-1}{r^2}\delta_{ij}+\sum_{k=1}^{K}h_{ik}h_{kj},\mb{ if }
1\leq i,j\leq K,\\
R_{ij}=&-\fr{m-K-1}{r^2}\delta_{ij},\mb{ if } K+1\leq
i,j\leq m,\\
R_{ij}=&0,\ \mb{if $1\leq i\leq K$, $K+1\leq j\leq m$, or
$K+1\leq i\leq m$, $1\leq j\leq K$}
\end{align}
which implies that the normalized scalar curvature of $g$ is given by
\begin{equation}
\kappa=\fr{m(K(K-1)-(m-K)(m-K-1))+(m-1)r^2}{m^2(m-1)r^2}.
\end{equation}
Thus
\begin{equation}
\label{ex1k}
\fr1{2m}(m^2\kappa-1)=\fr{K(K-1)-(m-K)(m-K-1)}{2(m-1)r^2}.
\end{equation}

Since $m\geq 3$, it follows from \eqref{Rij} and \eqref{ex1B}--\eqref{ex1k}
that the Blaschke tensor of $\td x$ is given by $A=\sum
A_{ij}\omega^i\omega^j$, where
\begin{align}
A_{ij}=&\fr{1}{2r^2}\delta_{ij},\mb{ if }
1\leq i,j\leq K,\quad A_{ij}=-\fr{1}{2r^2}\delta_{ij},\mb{ if } K+1\leq
i,j\leq m,\\[3pt]
A_{ij}=&0,\ \mb{if $1\leq i\leq K$, $K+1\leq j\leq m$, or
$K+1\leq i\leq m$, $1\leq j\leq K$}.
\end{align}

Clearly, $A$ has two distinct eigenvalues
$$
\lambda_1=-\lambda_2=\fr{1}{2r^2},
$$
which are constant. Thus by \eqref{exm1}, $A$ is parallel.

{\expl\label{expl3.3}\rm Given $r>0$. For any integers $m$ and $K$ satisfying $m\geq
3$ and $2\leq K\leq m-1$, let
$$\td y:M^{K}_1\to
\mathbb{H}^{K+1}_1\left(-\fr1{r^2}\right)\subset\bbr^{K+2}_2$$
be a regular and maximal space-like hypersurface with constant scalar curvature}
\begin{equation}
\label{ex2S}
\td S_1=\fr{-mK(K-1)+(m-1)r^2}{mr^2}
\end{equation}
and
$$
\td y_2: \mathbb{S}^{m-K}(r)\to\bbr^{m-K+1}
$$
be the canonical embedding. Set
\begin{equation}
\td M^m=M^{K}_1\times \mathbb{S}^{m-K}(r),\quad\td Y=(\td y,\td
y_2).
\end{equation}
Then $\lagl \td Y,\td Y\ragl_2=0$. Thus we have an immersion $\td Y:M^m\to \mathbb{C}^{m+2}\subset\bbr^{m+3}_2$ with the induced metric
$$g=\lagl d\td Y,d\td Y\ragl_2=\lagl d\td y,d\td y\ragl_2+d\td y_2\cdot d\td y_2,$$
which is certainly positive definite. It follows that, as Riemannian manifolds
\begin{equation}
\label{exm2}
(\td M^m,g)=(M_1,\lagl d\td y,d\td y\ragl_2)\times
\left(\mathbb{S}^{m-K}(r),d\td y^2_2\right).
\end{equation}

If we write $\td y=(\td y_0,\td y'_1,\td y''_1)\in\bbr^1_1\times\bbr^1_1\times\bbr^K\equiv\bbr^{K+2}_2$, then $\td y_0$ and $\td y'_1$ can not be zero simultaneously. So, without loss of generality, we can assume that $\td y_0\neq0$. In this case, we denote $\veps=\sgn(\td y_0)$ and write $\td y_1:=(\td y'_1,\td y''_1)$. Define
\begin{equation}
\td x_1=\fr{\td y_1}{\td y_0},\quad \td
x_2=\fr{\td y_2}{\td y_0}, \quad \td x=\varepsilon(\td x_1,\td x_2).
\end{equation}
Then $\td x\in \bbr^{m+2}_1,\td x^2=1$ and, similar to that in Example \ref{expl3.2}, $\td x:\td M^m\to \mathbb{S}^{m+1}_1$ defines a regular space-like hypersurface. In fact, since
\begin{equation}
\varepsilon d\td x=-\fr{d\td y_0}{\td y^2_0}(\td y_1,\td y_2) +\fr1{\td
y_0}(d\td y_1,d\td y_2),
\end{equation}
the induced metric $\td g=d\td x\cdot d\td x$ is related to $g$ by
\begin{align}
\td g=&\td y^{-4}_0d\td y^2_0(\lagl \td y_1,\td y_1\ragl_1+\td y_2\cdot\td y_2)+\td
y^{-2}_0(\lagl d\td y_1,d\td y_1\ragl_1+d\td y_2\cdot d\td y_2)\nnm\\
&-2\td y^{-3}_0d\td y_0(\lagl \td y_1,d\td y_1\ragl_1+ \td y_2\cdot d\td y_2)\nnm\\
=&\td y^{-2}_0(-d\td y^2_0+\lagl d\td y_1,d\td y_1\ragl_1+\td y_2\cdot d\td y_2)\nnm\\
=&\td y^{-2}_0g.
\end{align}

Suitably choose the time-like unit normal vector field  $(\td n_0,\td n_1)$ of $\td y$, define
$$\td n=(\td n_1,0)-\varepsilon\td
n_0\td x\in\bbr^{m+2}_1.$$
Then $\lagl \td n,\td n\ragl_1=-1,\lagl \td n,\td x\ragl_1=0,\lagl \td n,d\td x\ragl_1=0$ indicating that $\td n$ is a time-like unit normal vector field of $\td x$.
The second fundamental form $\td h$ of $\td x$ is given by
\begin{align}
\td h=&\lagl d\td n, d\td x\ragl_1=\lagl (d\td n_1,0)-\varepsilon d\td n_0\td x-\varepsilon\td n_0d\td x,d\td x\ragl_1\nnm\\
=&\lagl(d\td n_1,0),d\td x\ragl_1-\varepsilon d\td n_0\lagl \td x,d\td x\ragl_1-\varepsilon\td n_0\lagl d\td x,d\td x\ragl_1\nnm\\
=&\varepsilon(\td y^{-1}_0(-d\td n_0\cdot d\td y_0+\lagl d\td n_1,d\td y_1\ragl_1)-\td n_0\lagl d\td x,d\td x\ragl_1)\nnm\\
=&\varepsilon(\td y^{-1}_0\lagl d(\td n_0,\td n_1),d\td y\ragl_1-\td n_0\lagl d\td x,d\td x\ragl_1)\nnm\\
=&\varepsilon(\td y^{-1}_0h-\td n_0\td y^{-2}_0g)
\end{align}
where $h$ is the second fundamental form of $\td y$.

Let $\{E_i\,;1\leq i\leq K\}$ (resp. $\{E_i\,;K+1\leq i\leq
m\}$) be a local orthonormal frame field on $(M_1,d\td y^2)$ (resp.
on $\mathbb{S}^{m-K}(r)$). Then $\{E_i\,;1\leq i\leq m\}$ is a local
orthonormal frame field on $(M^m,g)$. Put $e_i=\varepsilon\td y_0E_i$,
$i=1,\cdots,m$. Then $\{e_i\,;1\leq i\leq m\}$ is a local
orthonormal frame field with respect to the metric $\td g=\lagl d\td x,d\td x\ragl_1$. Thus
\begin{equation}
\td h_{ij}=\td h(e_i,e_j)=\td y^2_0\td h(E_i,E_j)=\begin{cases}\varepsilon(\td y_0h_{ij}-\td n_0\delta_{ij}),\quad &\mb{when\ } 1\leq i,j\leq K,\\
-\varepsilon\td n_0g(E_i,E_j)=-\varepsilon\td n_0\delta_{ij},\quad &\mb{when\ }K+1\leq i,j\leq m,\\
0,\quad &\mb{for other}\quad i,j.
\end{cases}
\end{equation}
By the maximality of $\td y_1$, the mean curvature of $\td x$ is
\begin{equation}
\td H=\fr1m\sum\td h_{ii}=\varepsilon\fr{1}{m}(\td y_0KH_1-K\td n_0)-\varepsilon\fr{1}{m}(m-K)\td n_0=-\varepsilon\td n_0
\end{equation}
and
\begin{equation}
|\td h|^2=\sum^{K}_{i,j=1}\td y^2_0h^2_{ij}+\td n^2_0\delta^2_{ij}-2\td n_0\td y_0h_{ij}\delta_{ij}+\sum^m_{i,j=K+1}(-\td n_0)^2\delta^2_{ij}=\td y_0h^2+m\td n^2_0.
\end{equation}
Therefore, by definition, the conformal factor $\td \rho$ of $\td x$ is
determined by
$$\td \rho^2=\fr{m}{m-1}\left(\sum_{i,j}\td h^2_{ij}-m|\td
H|^2\right)=\fr{m}{m-1}\td y^2_0\sum_{i,j}h_{ij}^2 =\td y^2_0,$$
where we have used the Gauss equation and \eqref{ex2S}. Hence $\td \rho=|\td y_0|=\varepsilon\td y_0>0$ and thus $\td Y=\td \rho(1,\td x)$ is the conformal position vector of $\td x$. Consequently, the conformal metric of $\td x$ is defined by $\lagl d\td Y,d\td Y\ragl_2=g$. Furthermore, the conformal second fundamental form of $\td x$ is given by
\begin{equation}
\label{ex2B}
\td B=\td \rho^{-1}(\td h_{ij}-\td H\delta_{ij})\omega^i\omega^j
=\sum_{i,j=1}^{K}h_{ij}\omega^i\omega^j,
\end{equation}
where $\{\omega^i\}$ is the local coframe field on $M^m$ dual to $\{E_i\}$.

On the other hand, by \eqref{exm2} and the Gauss equations of $\td y_1$
and $\td y$, one finds the Ricci tensor of $g$ as
follows:
\begin{align}
R_{ij}=&-\fr{K-1}{r^2}\delta_{ij}+\sum_{k=1}^{K}h_{ik}h_{kj},\quad\mb{if\
} 1\leq i,j\leq K,\\
R_{ij}=&\fr{m-K-1}{r^2}\delta_{ij},\quad\mb{if\ } K+1\leq i,j\leq m,\\
R_{ij}=&0,\quad\mb{if $1\leq i\leq K$, $K+1\leq j\leq m$, or
$K+1\leq i\leq m$, $1\leq j\leq K$},
\end{align}
which implies that the normalized scalar curvature of $g$ is given by
\begin{equation}
\kappa=\fr{m((m-K)(m-K-1)-K(K-1))+(m-1)r^2}{m^2(m-1)r^2}.
\end{equation}
Thus
\begin{equation}
\label{ex2k}
\fr1{2m}(m^2\kappa-1)=\fr{(m-K)(m-K-1)-K(K-1)}{2(m-1)r^2}.
\end{equation}

Since $m\geq 3$, it follows from \eqref{Rij} and \eqref{ex2B}--\eqref{ex2k}
that the Blaschke tensor of $\td x$ is given by $A=\sum A_{ij}\omega^i\omega^j$, where
\begin{align}
A_{ij}=&-\fr{1}{2r^2}\delta_{ij},\quad\mb{if\ } 1\leq i,j\leq K,
\quad
A_{ij}=\fr{1}{2r^2}\delta_{ij},\quad\mb{if\ } K+1\leq i,j\leq m,
\\
A_{ij}=&0,\quad\mb{if $1\leq i\leq K$, $K+1\leq j\leq m$, or
$K+1\leq i\leq m$, $1\leq j\leq K$},
\end{align}
which, once again, implies that $A$ is parallel with two distinct
eigenvalues
$$
\lambda_1=-\lambda_2=-\fr{1}{2r^2}.
$$

\section{Proof of the main theorem}

To make the argument more readable, we divide the proof into several
lemmas.

Let $x:M^m\rightarrow \mathbb{S}^{m+1}_1$ be a regular space-like hypersurface.
\begin{lem}\label{lem4.1} If the Blaschke tensor $A$ is parallel, then the
conformal form $\Phi$ vanishes identically.
\end{lem}

\begin{proof} For any given point $p\in M^m$, take an orthonormal frame field $\{E_i\}$ around $p$ with respect to the conformal metric $g$, such that
$B_{ij}(p)=B_i\delta_{ij}$.
Then it follows from \eqref{Aijk} that
$$
A_{ijk}-A_{ikj}=B_{ij}\Phi_k-B_{ik}\Phi_j.
$$
Since $A$ is parallel, $A_{ijk}=0$ for any $i,j,k$. Thus at the
given point $p$, we have
\begin{equation}
B_i(\delta_{ij}\Phi_k(p)-\delta_{ik}\Phi_j(p))=0.
\end{equation}
By \eqref{trB}, there are different indices $i_1,i_2$ such that
$B_{i_1}\neq 0$ and $B_{i_2}\neq 0$. Then for any indices $i,j$, we
have
\begin{equation}
\label{phi0}
\delta_{i_1j}\Phi_i(p)-\delta_{i_1i}\Phi_j(p)=0,\quad
\delta_{i_2j}\Phi_i(p)-\delta_{i_2i}\Phi_j(p)=0.
\end{equation}
If $i=i_1$, put $j=i_2$; if $i\neq i_1$, put $j=i_1$. Then it
follows from \eqref{phi0} that $\Phi_i(p)=0$. By the arbitrariness of $i$
and $p$, we obtain that $\Phi\equiv 0$.\end{proof}
\begin{rmk}\rm
Since $A$ is parallel, then all eigenvalues of the
Blaschke tensor $A$ of $x$ are constant on $M^m$.
From the equation
$$
0=\sum A_{ijk}\omega^k=dA_{ij}-A_{kj}\omega^k_i-A_{ik}\omega^k_j,
$$
we obtain that
\begin{equation}
\label{wij0}
 \omega^i_j=0\quad \mb{in case that\ }A_i\neq A_j.
\end{equation}
\end{rmk}

\begin{lem}\label{lem4.2} If $A$ is parallel, then $B_{ij}=0$ as long as
$A_i\neq A_j$.\end{lem}

\begin{proof} Since $A$ is parallel, there exists around each
point a local orthonormal frame field $\{E_i\}$ such that
\begin{equation}
A_{ij}=A_i\delta_{ij}.
\end{equation}
It follows from \eqref{Phiij} and Lemma 4.1 that $\sum
B_{ik}A_{kj}-A_{ik}B_{kj}=\Phi_{ij}-\Phi_{ji}=0$. Then we have $B_{ij}(A_j-A_i)=0$.\end{proof}

Now, let $t$ be the number of the distinct eigenvalues of $A$, and
$\lambda_1,\cdots,\lambda_t$ denote the distinct eigenvalues of $A$.
Fix a suitably chosen orthonormal frame field $\{E_i\}$ for which
the matrix $(A_{ij})$ can be written as
\begin{equation}
\label{DiaA}
(A_{ij})=\mb{Diag}(\undbc{\lambda_1,\cdots,\lambda_1}_{k_1},
\undbc{\lambda_2,\cdots,\lambda_2}_{k_2},\cdots,\undbc{\lambda_t,\cdots,
\lambda_t}_{k_t}),
\end{equation}
or equivalently,
\begin{equation}
A_1=\cdots=A_{k_1}=\lambda_1,\cdots,
A_{m-k_t+1}=\cdots=A_m=\lambda_t.
\end{equation}
\begin{lem}\label{lem4.3} Suppose that $t\geq 3$. If, with respect to an orthonormal
frame field $\{E_i\}$, \eqref{DiaA} holds and at a point $p$,
$B_{ij}=B_i\delta_{ij}$, then
$$
B_i=B_j\quad\mb{in the case that\ }A_i=A_j.
$$\end{lem}

\begin{proof}  By \eqref{wij0}, for any $i,j$ satisfying $A_i\neq A_j$,
we have $\omega^i_j=0$. Differentiating this equation, we obtain
from \eqref{Rijkl} that
$$
0=R_{ijji}=B_{ij}^2-B_{ii}B_{jj} +A_{ii}-A_{ij}\delta_{ij}
+A_{jj}-A_{ij}\delta_{ij}.
$$
Thus at $p$, it holds that
\begin{equation}
\label{BiBj}
-B_iB_j +A_i+A_j=0.
\end{equation}

If there exist indices $i,j$ such that $A_i=A_j$ but $B_i\neq B_j$,
then for all $k$ satisfying $A_k\neq A_i$, we have
\begin{equation}
\label{1}
-B_iB_k +A_i+A_k=0,\quad -B_jB_k +A_j+A_k=0.
\end{equation}
It follows from \eqref{1} that $(B_i-B_j)B_k=0$, which implies that
$B_k=0$. Thus by \eqref{1}, we obtain $A_k=-A_i=-A_j$. This implies
that $t=2$, contradicting the assumption.\end{proof}

\begin{cor}\label{cor4.4} If $t\geq 3$, then there exists an orthonormal
frame field $\{E_i\}$ such that
\begin{equation}
\label{DADB}
A_{ij}=A_i\delta_{ij},\quad B_{ij}=B_i\delta_{ij}.
\end{equation}
Furthermore, if \eqref{DiaA} holds, then
\begin{equation}
\label{DiaB}
(B_{ij})=\rm{Diag}(\undbc{\mu_1,\cdots,\mu_1}_{k_1},
\undbc{\mu_2,\cdots,\mu_2}_{k_2},\cdots,\undbc{\mu_t,\cdots,
\mu_t}_{k_t}),
\end{equation}
that is
\begin{equation}
B_1=\cdots=B_{k_1}=\mu_1,\cdots,
B_{m-k_t+1}=\cdots=B_m=\mu_t.
\end{equation}
\end{cor}

\begin{proof} Since $A$ is parallel, we can find a local orthonormal
frame field $\{E_i\}$, such that \eqref{DiaA} holds. It then suffices to show that, at any point, the component matrix
$(B_{ij})$ of $B$ with respect to $\{E_i\}$ is diagonal. Note that
$k_1,\dots,k_t$ are the multiplicities of the eigenvalues
$\lambda_1,\cdots,\lambda_t$, respectively. By Lemma \ref{lem4.2}, we can
write
$$
(B_{ij})=\text{Diag}(B_{(1)},\cdots,B_{(t)}),
$$
where $B_{(1)},\cdots,B_{(t)}$ are square matrices of orders
$k_1,\cdots,k_t$, respectively. For any point $p$, we can choose a
suitable orthogonal matrix $T$ of the form
$$
T=\text{Diag}(T_{(1)},\cdots,T_{(t)}),
$$
with $T_{(1)},\cdots,T_{(t)}$ being orthogonal matrices of orders
$k_1,\cdots,k_t$, such that $$T\cdot(B_{ij}(p))\cdot
T^{-1}=\text{Diag}(B_1,\cdots,B_m),$$ where $B_1,\cdots,B_m$ are the
eigenvalues of tensor $B$ at $p$. It then follows from Lemma \ref{lem4.3}
that
$$
B_1=\cdots=B_{k_1}:=\mu_1,\cdots,B_{m-k_t+1}=\cdots=B_m:=\mu_t.
$$
Hence
\begin{align}
&T_{(1)}B_{(1)}(p)T^{-1}_{(1)}=\text{Diag}(\mu_1,\cdots,\mu_1),\\
&\ \ \cdots\cdots\\
&T_{(t)}B_{(t)}(p)T^{-1}_{(t)}=\text{Diag}(\mu_t,\cdots,\mu_t).
\end{align}
Therefore
\be
B_{(1)}(p)=T^{-1}_{(1)}\cdot\text{Diag}(\mu_1,\cdots,\mu_1)\cdot
T_{(1)}=\text{Diag}(\mu_1,\cdots,\mu_1).
\ee
In the same way,
\be
B_{(2)}(p)=\text{Diag}(\mu_2,\cdots,\mu_2),\cdots,
B_{(t)}(p)=\text{Diag}(\mu_t,\cdots,\mu_t). \ee
Thus
$$(B_{ij}(p))=\text{Diag}(\mu_1,\cdots,\mu_1,\cdots,\mu_t,\cdots,\mu_t).$$
\end{proof}

\begin{lem}\label{lem4.5} If $t\geq 3$, then all the conformal principal curvatures
$\mu_1,\cdots,\mu_t$ of $x$ are constant, and hence $x$ is conformal
isoparametric.\end{lem}

\begin{proof} Without loss of generality, we only need to show that $\mu_1$ is constant.
To this end, choose a frame field $\{E_i\}$ such that \eqref{DiaA} and
\eqref{DiaB} hold. Note that, by \eqref{wij0}, when $1\leq i\leq k_1$ and
$j>k_1$, we have
$$
\sum B_{ijk}\omega^k=dB_{ij}-\sum B_{kj}\omega^k_i-\sum
B_{ik}\omega^k_j=0,
$$
which implies that $B_{ijk}=0$.

By Lemma \ref{lem4.1}, $\Phi\equiv 0$. Hence from \eqref{Bijk} one seen that $B_{ijk}$ is symmetric with respect to $i,j,k$. It follows that
$B_{ijk}=0$, in case that two indices in $i,j,k$ are less than
or equal to $k_1$ with the other index larger than $k_1$, or one
index in $i,j,k$ is less than or equal to $k_1$ with the other two
indices larger than $k_1$. In particular, for any $i,j$ satisfying
$1\leq i,j\leq k_1$,
$$
\sum_{k=1}^{k_1} B_{ijk}\omega^k=dB_{ij}-\sum
B_{kj}\omega^k_i-\sum
B_{ik}\omega^k_j=dB_i\delta_{ij}-B_j\omega^j_i-B_i\omega^i_j.
$$
Putting $j=i$, one obtains
\begin{equation}
\sum_{k=1}^{k_1} B_{iik}\omega^k=d\mu_1,
\end{equation}
which implies that
\begin{equation}
\label{Eu10}
E_k(\mu_1)=0,\quad k_1+1\leq k\leq m.
\end{equation}
Similarly,
\begin{equation}
\label{EB0}
E_i(B_j)=0,\quad 1\leq i\leq k_1,\ k_1+1\leq j\leq m.
\end{equation}

On the other hand, we see from \eqref{BiBj} that
\begin{equation}
\label{2}
-\mu_1B_j+\lambda_1+A_j=0, \quad k_1+1\leq j\leq m,
\end{equation}
hold identically. Differentiating \eqref{2} in the direction of $E_k$,
$1\leq k\leq k_1$, and using \eqref{EB0}, we obtain
$$
E_k(\mu_1)B_j=0,\quad 1\leq k\leq k_1,\ k_1+1\leq j\leq m.
$$
By \eqref{trB} there exists some index $j$
such that $k_1+1\leq j\leq m$ and $B_j\neq 0$. Therefore,
$E_k(\mu_1)=0$ for $1\leq k\leq k_1$. This together with \eqref{Eu10}
implies that $\mu_1$ is a constant.\end{proof}

\begin{cor} If $t\geq 3$, then $t=3$ and $B$ is
parallel.\end{cor}

\begin{proof} Indeed, the conclusion that $B$ is parallel comes from
\eqref{wij0}, Corollary \ref{cor4.4} and Lemma \ref{lem4.5}.

If $t>3$, then there exist at least four indices $i_1,i_2,i_3,i_4$,
such that $A_{i_1},A_{i_2},A_{i_3},A_{i_4}$ are distinct each other.
Then it follows from \eqref{BiBj} that
\begin{align}
&-B_{i_1}B_{i_2} +A_{i_1}+A_{i_2}=0,\quad -B_{i_3}B_{i_4}
+A_{i_3}+A_{i_4}=0,\\
&-B_{i_1}B_{i_3} +A_{i_1}+A_{i_3}=0,\quad - B_{i_2}B_{i_4}
+A_{i_2}+A_{i_4}=0.
\end{align}
Consequently, we obtain $(A_{i_1}-A_{i_4})(A_{i_2}-A_{i_3})=0$, a
contradiction.\end{proof}

\begin{lem} If $t\leq 2$ and $B$ is not parallel, then one of
the following cases holds:

\hs{.6cm}$(1)$ $t=1$ and $x$ is conformal isotropic.

\hs{.6cm}$(2)$ $t=2$, $\lambda_1+\lambda_2=0$ and $B_i=0$ either for
all $1\leq i\leq k_1$, or for all $k_1+1\leq i\leq m$.\end{lem}

\begin{proof} Note that $\Phi\equiv 0$. Thus $x$ is conformal isotropic if and only if $t=1$.

If $t=2$, then for any point $p\in M^m$, we can find an orthonormal frame field
$\{E_i\}$ such that \eqref{DADB} holds at $p$.

By \eqref{wij0}, we see that
\begin{equation}
\label{3}
\omega^i_j=0,\quad 1\leq i\leq k_1,\quad k_1+1\leq j\leq m,
\end{equation}
hold identically. Taking exterior
differentiation of \eqref{3} and making use of \eqref{Rijkl}, we find that,
at $p$
\begin{equation}
\label{4}
-B_iB_j+A_i+A_j=0,\quad 1\leq i\leq k_1,\quad k_1+1\leq j\leq m.
\end{equation}
If there exist one pair of indices $i_0,j_0$ satisfying $1\leq
i_0\leq k_1$, $k_1+1\leq j_0\leq m$ such that $B_{i_0}\neq 0$ and
$B_{j_0}\neq 0$, then for each index $i$ satisfying $1\leq i\leq
k_1$, we obtain
$$
-B_{i_0}B_{j_0}+A_{i_0}+A_{j_0}=0,\quad -B_iB_{j_0}+A_i+A_{j_0}=0,
$$
from which it follows that $(B_i-B_{i_0})B_{j_0}=0$, or
equivalently
$$B_i=B_{i_0},\quad 1\leq i\leq k_1.$$
Similarly, we obtain
$$B_j=B_{j_0},\quad k_1+1\leq j\leq m.$$
Consequently, \eqref{DiaB} also holds in the case that $t=2$. Now, an argument similar to that in the proof of Lemma \ref{lem4.5} shows that the conformal principal curvatures $B_i$ are all constant. Therefore $B$ is
parallel by \eqref{3}, contradicting to the assumption. Thus either
$B_i=0$ for all indices $i$ satisfying $1\leq i\leq k_1$, or $B_j=0$
for all indices $j$ satisfying $k_1+1\leq j\leq m$. In both cases we
have, by \eqref{4}, $\lambda_1+\lambda_2=0$.\end{proof}

\textsc{Proof of Theorem 1.2.} By Theorem \ref{nie1} and Theorem \ref{nie2}, it clearly
suffices to consider the case that $x$ neither is conformal isotropic nor
has parallel conformal second fundamental form. Hence from those Lemmas
proved in this section, we can suppose without loss of generality
that
\begin{equation}
\label{t2}
t=2,\quad \lambda_1=-\lambda_2=\lambda\neq 0,\quad
B_{k_1+1}=\cdots=B_m=0.
\end{equation}
Since $\sum B_i=0$ and $\sum B^2_i=(m-1)/m$, one sees easily that
$m\geq 3$. Since $A$ is parallel, the tangent bundle $TM^m$ of $M^m$
has a decomposition $TM^m=V_1\oplus V_2$, where $V_1$ and $V_2$ are
the eigenspaces of $A$ corresponding to the eigenvalues
$\lambda_1=\lambda$ and $\lambda_2=-\lambda$, respectively.

Let $\{E_i\,;1\leq i\leq k_1\}$ and $\{E_j\,;k_1+1\leq j\leq m\}$ be
orthonormal frame fields for subbundles $V_1$ and $V_2$,
respectively. Then $\{E_i\,;1\leq i\leq m\}$ is an orthonormal frame
field on $M^m$ with respect to the conformal metric $g$. Then \eqref{3}
implies that both $V_1$ and $V_2$ are integrable, and thus
Riemannian manifold $(M^m,g)$ can be locally decomposed into a
direct product of two Riemannian manifolds $(M_1,g_1)$ and
$(M_2,g_2)$, that is, as a Riemannian manifold, locally
\begin{equation}
\label{Mg}
(M^m,g)=(M_1,g_1)\times (M_2,g_2).
\end{equation}
It follows from \eqref{Rijkl}, \eqref{DiaA}, \eqref{t2} and \eqref{Mg} that the Riemannian
curvature tensors of $(M_1,g_1)$ and $(M_2,g_2)$ have the following
components, respectively,
\begin{equation}
\label{M1R}
R_{ijkl}=2\lambda (\delta_{il}\delta_{jk}-\delta_{ik}\delta_{jl})
+(B_{ik}B_{jl}-B_{il}B_{jk}),\quad 1\leq
i,j,k,l\leq k_1,
\end{equation}
\begin{equation}
\label{M2R}
R_{ijkl}=-2\lambda
(\delta_{il}\delta_{jk}-\delta_{ik}\delta_{jl}),\quad k_1+1\leq
i,j,k,l\leq m.
\end{equation}
Thus $(M_2,g_2)$ is of constant sectional curvature $-2\lambda$.

Next we consider the following cases separately.

Case (1): $\lambda>0$. In this case, set $r=(2\lambda)^{-1/2}$. Then
$(M_2,g_2)$ can be locally identified with
$\mathbb{H}^{m-k_1}\left(-\fr1{r^2}\right)$. Let $\td y=(\td y_0,\td
y_2):\mathbb{H}^{m-k_1}\left(-\fr1{r^2}\right)\to\bbr^{m-k_1+1}_1$ be the
canonical embedding.

Since $h=\sum_{i,j=1}^{k_1}B_{ij}\omega^i\omega^j$ is a Codazzi
tensor on $(M_1,g_1)$, it follows from \eqref{M1R} that there exists a
maximal immersed hypersurface
$$\td y_1:(M_1,g_1)\to
\mathbb{S}^{k_1+1}_1(r)\subset\bbr^{k_1+2}_1,\quad 2\leq k_1\leq m-1,$$
which has $h$ as its second fundamental form. Clearly, $\td y_1$  has constant scalar curvature
$$S_1=\fr{mk_1(k_1-1)+(m-1)r^2}{mr^2},$$
and $M^m$ can be locally identified with $\td M^m=(M_1,g_1)\times \mathbb{H}^{m-k_1}(-\fr1{r^2})$.

Define $\td x_1={\td y_1}/{\td y_0}$, $\td x_2={\td y_2}/{\td y_0}$
and $\td x=(\td x_1,\td x_2)$. Then, by the discussion in Example
3.2, $\td x:\td M^m\to \mathbb{S}^{m+1}_1$ yields a regular space-like hypersurface with
the given $g$ and $B$ as its conformal metric and conformal second fundamental
form, respectively. Therefore, by Theorem \ref{thm2.4}, $x$ is conformal equivalent
to $\td x$.

Case (2): $\lambda<0$. In this case, set $r=(-2\lambda)^{-1/2}$,
then $(M_2,g_2)$ can be locally identified with $\mathbb{S}^{m-k_1}(r)$. Let
$\td y_2:\mathbb{S}^{m-k_1}(r)\to\bbr^{m-k_1+1}$ be the canonical embedding.

Since $h=\sum_{i,j=1}^{k_1}B_{ij}\omega^i\omega^j$ is a Codazzi
tensor on $(M_1,g_1)$, it follows from \eqref{M1R} that there exists a
maximal immersed hypersurface
$$\td y=(\td y_0,\td y_1):(M_1,g_1)\to
\mathbb{H}^{k_1+1}_1\left(-\fr1{r^2}\right)\subset\bbr^{k_1+2}_2,\quad 2\leq
k_1\leq m-1,$$
which has $h$ as its second fundamental form. Clearly, $\td y$ has constant scalar curvature
$$S_1=\fr{-mk_1(k_1-1)+(m-1)r^2}{mr^2},$$
and $M^m$ can be locally identified with $\td M^m=(M_1,g_1)\times
\mathbb{S}^{m-k_1}\left(r\right)$.

Assume without loss of generality that $\td y_0\neq 0$. Define $\veps=\sgn(\td y_0)$ and let $\td x_1=\veps{\td y_1}/{\td y_0}$, $\td x_2=\veps{\td y_2}/{\td y_0}$
and $\td x=(\td x_1,\td x_2)$. Then, by the discussion in Example
3.3, $\td x:\td M^m\to \mathbb{S}^{m+1}_1$ defines a regular space-like hypersurface
with the given $g$ and $B$ as its conformal metric and conformal second
fundamental form, respectively. It follows by Theorem \ref{thm2.4} that $x$ is
conformal equivalent to $\td x$.\hfill $\Box$

\end{document}